\documentclass{amsart}

\usepackage{amsfonts}
\usepackage{verbatim}
\usepackage{amssymb}
\usepackage{amscd}
\usepackage{amsmath}
\usepackage{amsthm}
\usepackage{graphicx}
\usepackage{subfigure}
\usepackage{geometry}
\usepackage{cite}
\usepackage{ifpdf}
\ifpdf
\fi

\newtheorem{theorem}{Theorem}[section]

\newtheorem{lemma}[theorem]{Lemma}

\newtheorem{corollary}[theorem]{Corollary}
\newtheorem{definition}[theorem]{Definition}

\begin{document}
\bibliographystyle{plain}

\title{A note on the shameful conjecture}
\author{Sukhada Fadnavis}
\email{Sukhada Fadnavis <sukhada@math.harvard.edu>}
\address{Dept. of Mathematics, Harvard University, One Oxford Street, Cambridge, MA 02138.}
\date{\today}
\maketitle

\begin{abstract}
	Let $P_G(q)$ denote the chromatic polynomial of a graph $G$ on $n$ vertices. The `shameful conjecture' due to Bartels and Welsh states that, $$\frac{P_G(n)}{P_G(n-1)} \geq \frac{n^n}{(n-1)^n}.$$ Let $\mu(G)$ denote the expected number of colors used in a uniformly random proper $n$-coloring of $G$. The above inequality can be interpreted as saying that $\mu(G) \geq \mu(O_n)$, where $O_n$ is the empty graph on $n$ nodes. This conjecture was proved by F. M. Dong, who in fact showed that,  
$$\frac{P_G(q)}{P_G(q-1)} \geq \frac{q^n}{(q-1)^n}$$ for all $q \geq n$. There are examples showing that this inequality is not true for all $q \geq 2$. In this paper, we show that the above inequality holds for all $q \geq 36D^{3/2}$, where $D$ is the largest degree of $G$. It is also shown that the above inequality holds true for all $q \geq 2$ when $G$ is a claw-free graph. 
\end{abstract}

\section{Introduction} 

The chromatic polynomial is an important algebraic object studied in the field of graph coloring. For a graph $G = (V, E)$ with vertex set $V$ and edge set $E$, let $\sigma: V \rightarrow \{1, \ldots , q \}$ be a map. The map $\sigma$ is said to be a proper $q$-coloring of graph $G$ if for every edge in $E$ the endpoints of the edge have distinct images under $\sigma$. Let $P_G(q)$ denote the total number of proper $q$-colorings of $G$. It is well known that $P_G(q)$ is a polynomial in $q$ and is known as the chromatic polynomial. In fact, 
	\begin{equation}
   		 P_G(q) = \sum_{E' \subseteq E} (-1)^{|E'|} q^{C(E')},
	\end{equation}
where $C(E')$ denotes the number of connected components of $(V,E')$. The sum goes over all subsets $E' \subseteq E$ of the edge set $E$. This is easily seen using the inclusion-exclusion principle as is explained in section \ref{method}. \\

Properties of the chromatic polynomial have been studied extensively. For example, the log-concavity of the chromatic polynomial (proved for its coefficients \cite{JuneHuh}) is well studied \cite{Brenti, Brenti2, Stanley3}. There has also been a lot of interest in understanding the roots of the chromatic polynomial \cite{Brown, BRW, Sokal, Borgs, Fernandez}. \\

Bartels and Welsh \cite{BartlesWelsh} studied a Markov chain on colorings, which would help approximate $\mu(G)$, the expected number of colors in a uniformly random proper  $n$-coloring on $G$. They proved that,
\begin{equation}
	\mu(G) = n\left(1 - \frac{P_G(n-1)}{P_G(n)} \right),
\end{equation} 
where $n = |V|$. 

They conjectured that on the set of graphs on $n$ nodes, $\mu$ is minimized when $G$ is the empty graph, $O_n$, on $n$ vertices, that is, 
\begin{equation}
	n\left(1 - \frac{P_G(n-1)}{P_G(n)} \right) = \mu(G) \geq \mu(O_n) =  n\left(1 - \left(\frac{n-1}{n}\right)^n \right).
\end{equation} 

Rewriting this, the conjecture states:
\begin{equation}
	\frac{P_G(n)}{P_G(n-1)} \geq \frac{n^n}{(n-1)^n}.
\end{equation}

They point out that the inequality 
\begin{equation}\label{geninequality}
	\frac{P_G(q)}{P_G(q-1)} \geq \frac{q^n}{(q-1)^n}.
\end{equation}
may not hold true for all $q \geq 2$, as was shown by the following example due to Colin McDiarmid. Let $G = K_{n,n}$, the complete bipartite graph with partitions of size $n$ each with $n \geq 10$. Then,

\begin{equation}
	\frac{P_{G}(3)}{P_G(2)} = \frac{6+6(2^n-2)}{2} = \frac{3(2^n-1)}{1} < \frac{3^{2n}}{2^{2n}}.
\end{equation}

This conjecture came to be dubbed as the `shameful conjecture' and was proved 5 years later by Dong \cite{Dong}.  In fact, Dong proved the stronger result that inequality \ref{geninequality} holds true for all $q \geq n-1$.
In particular, this implies, 
\begin{equation}
	\frac{P_G(n)}{P_G(n-1)} \geq \frac{n^n}{(n-1)^n} \geq e.
\end{equation}

Before that Seymour \cite{Seymour} also showed that, 
\begin{equation}
	\frac{P_G(n)}{P_G(n-1)}  \geq \frac{685}{252} = 2.7182539... < e.
\end{equation}

In this paper, the following complementary result to Dong's result is proved:

\begin{theorem}\label{Shameful}
	Let $G = (V,E)$ be a finite simple graph. Let $D$ denote the largest degree of $G$. Then, for $q > 36 D^{3/2}$ we have, 
\begin{equation}
	\frac{P_G(q+1)}{P_G(q)}  \geq \frac{(q+1)^n}{q^n}.
\end{equation}
\end{theorem}

We conjecture that this result can be improved to be true for $q > C D$ for some constant $C$, but do not have a proof of that as yet. \\

It turns out that the situation is much nicer for a class of graphs called claw-free graphs:
\begin{definition} A claw is the bipartite graph $K_{1,3}$. A graph is said to be claw-free if it does not have any induced subgraphs isomorphic to $K_{1,3}$. \end{definition} 

Claw-free graphs have been studied in great detail and completely classified by Seymour and Chudnovsky \cite{SC1, SC2, SC3, SC4, SC5}.\\

Examples of clawfree graphs include complete graphs, cycles, complements and triangle-free graphs. Line graphs form an important class of claw-free graphs, which we will use later. Given a graph $G= (V,E)$, its line graph $L(G)$ is defined as follows. The vertex set of $L(G)$ consists of the edges $E$ of $G$. And there is an edge between $e_1, e_2 \in E$ if $e_1, e_2$ share a common end point. It is easy to see that line graphs are claw-free. \\

For claw-free graphs we prove:
\begin{theorem} If $G$ is a claw-free graph then, 
	\begin{equation}
		\frac{P_G(q+1)}{P_G(q)}  \geq \frac{(q+1)^n}{q^n}.
	\end{equation}
	for all integers $q \geq 1$. 
\end{theorem}

In fact, a generalization is discussed and proved in section \ref{clawfree}. \\

An interesting corollary for edge-colorings is obtained as a consequence of the above theorem. To discuss this corollary we consider a different interpretation of the above problem. Note that the inequality 
\begin{equation}
	\frac{P_G(q+1)}{P_G(q)}  \geq \frac{(q+1)^n}{q^n}
\end{equation}
can be re-written as,
\begin{equation}
	\frac{P_G(q+1)}{(q+1)^n}  \geq \frac{P_G(q)}{q^n}.
\end{equation}
Suppose the vertices of graph $G$ are independently and uniformly colored with one of $q$ colors. The quantity $P_G(q)/q^n$ denotes the probability that such a random coloring is a proper coloring of $G$.  Let $\sigma_q$ denote a uniformly chosen random $q$-coloring of $G$. Then the above inequality can be interpreted as saying that:
\begin{equation}
	\Pr (\sigma_{q+1} \text{ is proper}) \geq \Pr (\sigma_{q} \text{ is proper}).
\end{equation}

Thus, for a clawfree graph $G$, we have that the above inequality is true for all integers $q \geq 1$. \\

Now we consider the edge-coloring problem. An edge-coloring of a graph $G= (V,E)$ with $q$ colors is defined to be a map $f: E \rightarrow \{1, \ldots, q\}$. An edge-coloring of $G$ is said to be proper if no two intersecting edges are mapped to the same number. Let $p_G(E, q)$ denote the probability that a uniformly chosen random edge-coloring on $G$ with $q$ colors is proper. 
Then, we have the following corollary: 
\begin{corollary}
	For a finite, simple graph $G$, we have $p_G(E,q + 1) \geq p_G(E, q)$ for all integers $q \geq 1$. 
\end{corollary}
Thus, the case for edge colorings is nicer than the case for vertex colorings. In fact, the edge coloring problem is the same as the vertex coloring problem on the line graph of $G$. Line graphs are claw-free. Thus, the edge-coloring result immediately follows from the vertex coloring result for claw-free graphs. 

The remaining paper is organized as follows. Theorem \ref{Shameful} is proved in section \ref{method} and the proof of the claw-free graph case is provided in section \ref{clawfree}.  

\section{Proof of Theorem \ref{Shameful}}\label{method}

We begin by showing that the chromatic function is in fact a polynomial. As explained in the introduction, this can be seen using the inclusion-exclusion principal.  \\
For $E' \subseteq E$, let $N_{E'}$ denote the number of proper $q$-colorings of $G$ in which the edges in $E'$ are all monochromatic. Then, the inclusion exclusion principle gives:
\begin{equation}
	P_G(q) = \sum_{E' \subseteq E} (-1)^{|E'|} N_{E'}.
\end{equation}
Note that, $N_{E'} =  q^{C(E')}$ where $C(E')$ denotes the number of connected components of $(V,E')$.  Hence, 
\begin{equation}
   	P_G(q) = \sum_{E' \subseteq E} (-1)^{|E'|} q^{C(E')}.
\end{equation}

Sokal \cite{Sokal} and Borgs \cite{Borgs} provided a bound on the roots of the chromatic polynomial in terms of the largest degree $D$ of $G$. We state their result below:

\begin{theorem}\label{rootbound}[Sokal , Borgs] Let $G = (V,E)$ be a finite, simple graph with highest degree $D$. Let 
\begin{equation}\label{Kdefined}
	K(a) = \frac{a + e^a}{\log(1 + a e^{-a})},
\end{equation}
 and let $K^* = \min_{a \geq 0} K(a)$. Then, the roots of the chromatic polynomial $P_G(q)$ are bounded above in absolute value by $K^*D$.  Note that $K^* \leq 7.963907$. 
\end{theorem}

We use this bound on the roots of the chromatic polynomial to get bounds on the coefficients of its log expansion. Note that we assume $G$ to be connected henceforth since proving Theorem \ref{Shameful} for connected graphs implies the result for a disjoint union of connected graphs by the multiplicative property of the chromatic polynomial. 

\begin{theorem} If
	\begin{equation}\label{LogExpansion}
		\log \frac{P_G(q)}{q^n} = \sum_{N =1}^{\infty} c_N q^{-N}.
	\end{equation}
Then, $|c_k| \leq \frac{2|E|}{k} (8D)^k$
\end{theorem}
\begin{proof} Suppose the roots of $P_G(q)$ are $\alpha_1, \ldots, \alpha_n$.  Let $$f(q) = \log \frac{P_G(q)}{q^n} .$$ Then, 

\begin{equation}
	f(q) = \log \frac{\prod_{i=1}^n (q-\alpha_i)}{q^n} = \sum_{i=1}^n \log\left(1 - \frac{\alpha_i}{q} \right).
\end{equation}
By Sokal's theorem, $|\alpha_i| < 8D$, hence, for $q > 8D$ we have, 
\begin{equation}
	f(q) =  \sum_{i=1}^n \log\left(1 - \frac{\alpha_i}{q} \right) = - \sum_{k = 1}^{\infty} \frac{1}{k} \frac{1}{q^k} \sum_{i=1}^{n} \alpha_i^k.
\end{equation}
Thus,
\begin{equation}
	c_k = - \frac{1}{k} \sum_{i=1}^{n} \alpha_i^k,
\end{equation}
and hence, 
\begin{equation}
	|c_k| \leq \frac{1}{k} \sum_{i=1}^{n} |\alpha_i|^k \leq \frac{1}{k} n (8D)^k \leq \frac{2|E|}{k}(8D)^k.
\end{equation}
 The last inequality follows from assuming that $G$ is connected and $n \geq 2$, hence, $|E| \geq n/2$. 
\end{proof}

The next result finds the exact values of $c_1, c_2$. 

\begin{lemma}
In the above equation $c_1 = -|E|$ and $c_2 = -|T| - |E|/2$, where $|E|$ is the number of edges and $|T|$ is the number of triangles in $G$. 
\end{lemma}
\begin{proof} Let 
\begin{equation}
	P_G(q) = \sum_{E' \subseteq E} (-1)^{|E'|} q^{C(E')} = \sum_{i=0}^n a_i q^i,
\end{equation} 
where $C(E')$ denotes the number of connected components of $(V,E')$. Thus, $a_n = 1$ since $C(E') = n$ only when $E' = \emptyset$. Further, $a_{n-1} = -|E|$ since $C(E') = n-1$ only when $E'$ consists of a single edge. To compute $a_{n-1}$ we note that $C(E') = n-2$ only when $E'$ consists exactly of two edges (joint or disjoint) or three edges forming a triangle. So,\begin{equation}
	a_{n-2} = \binom{|E|}{2} - |T|. 
\end{equation}
We also know that, 
\begin{equation}
	a_{n-1} = - \sum_{i=1}^n \alpha_i, \text{ and } a_{n-2} = \sum_{1 \leq i < j \leq n} \alpha_i \alpha_j. 
\end{equation}
Thus, 
\begin{equation}
	\sum_{i=1}^n \alpha_i^2 = a_{n-1}^2 - 2a_{n-2} = |E|^2 - |E|(|E|-1) + 2|T| =  |E| + 2|T|. 
\end{equation}

As seen above, 
\begin{equation}
	c_k = - \frac{1}{k}\sum_{i=1}^n \alpha_i^k.
\end{equation}
Thus, 
$c_1 = -|E|$ and $c_2 =  - |T| - \frac{|E|}{2}. $

\end{proof}

Now we can prove the main result of this paper: 

\begin{proof}
	From the above theorems we have,
	\begin{equation}
        		\log \frac{P_G(q)}{q^n} = \sum_{N =1}^{\infty} c_N q^{-N} = -\frac{|E|}{q} - \frac{|T| + |E|/2}{q^2} + \sum_{N = 3}^{\infty} c_Nq^{-N}.
	\end{equation}
\end{proof}

Let $f(q) = \log \frac{P_G(q)}{q^n}$. Restricting to $q > 8D$ and taking the derivative gives, 
\begin{equation}
	\begin{split}
		& f'(q) = \frac{|E|}{q^2} + \frac{2|T| + |E|}{q^3} - \sum_{N = 3}^{\infty} Nc_Nq^{-N-1} \geq \frac{|E|}{q^2} + \frac{2|T| + |E|}{q^3} - \sum_{N = 3}^{\infty} 2|E|(8D)^Nq^{-N-1} \\
		& \geq \frac{|E|}{q^2} \left( 1 - \sum_{N = 3}^{\infty} 2(8D)^Nq^{-N+1} \right) = \frac{|E|}{q^2} \left( 1 - 16D\sum_{N = 3}^{\infty} \left(\frac{8D}{q}\right)^{N-1} \right)\\
		& = \frac{|E|}{q^2} \left( 1 - 16D \left(\frac{8D}{q}\right)^{2} \frac{1}{1-\frac{8D}{q}} \right).
	\end{split}
\end{equation}

When $q > 36D^{3/2}$, and $D \geq 2$ we have, 

\begin{equation}
		\frac{|E|}{q^2} \left( 1 - 16D \left(\frac{8D}{q}\right)^{2} \frac{1}{1-\frac{8D}{q}} \right)  \geq \frac{|E|}{q^2} \left( 1 - 			\frac{1024D^3}{q^2} \times \frac{6}{5} \right) >  \frac{|E|}{q^2} \left( 1 - \frac{1229D^3}{1296D^3}  \right) > 0. 
\end{equation} 

In the above we assumed that $D \geq 2$. This is OK since the statement of the theorem is easily checked to be true when $D = 1$. 
Thus, $P_G(q)/q^n$ is increasing when $q > 36D^{3/2}$.  This completes the proof. 

\section{The case of claw-free graphs}\label{clawfree}

Recall that the inequality
\begin{equation}
	\frac{P_G(q)}{q^n} \geq \frac{P_G(q-1)}{(q-1)^n}
\end{equation}
 can be interpreted as saying that the probability that a uniformly random $q$-coloring of $G$ is a proper coloring is greater than the probability that a uniformly random $(q-1)$-coloring of $G$ is a proper coloring. To prove the result for claw-free graphs we will consider a slightly more general situation.\\

Consider a graph $G$ on $n$ vertices.
Suppose the vertices are colored independently at random with $q$ colors occurring with probabilities $p_1, \ldots , p_q$. Let $P_G(p_1, \ldots, p_q)$ denote the probability that the random coloring thus obtained is a
proper coloring. Note that in this setting, 
\begin{equation}
    	\frac{P_G(q)}{q^n} =  P_G\left(\frac{1}{q}, \ldots, \frac{1}{q}\right).
\end{equation}

\begin{theorem}\label{clawfreeproof} If $G$ is claw-free then $P_G(p_1, \ldots p_q)$ is maximized when $p_1 = \cdots = p_q = 1/q$. In fact $P_G$ is Schur-concave on the set of probability distributions $\textbf{p} = (p_1,\ldots, p_q)$. In particular, 
\begin{equation}
	\frac{P_G(q)}{q^n} \geq \frac{P_G(q-1)}{(q-1)^n}
\end{equation}
for all $q \geq 2$. 
\end{theorem}
\begin{proof} Suppose we start with a distribution $p = (p_1, \ldots, p_q)$ on the colors $1, \ldots, q$. Fix $p_i$
for $i \geq 3$. Let $H \subseteq G$ be an induced subgraph of $G$. We denote by $C_H$ the set of all colorings of $H$ with colors $3$ to $q$.  Let $N(H, a_3, \ldots , a_q)$ denote the number of proper colorings of $H$
with $a_i$ vertices colored with color $i$. Note that $N(H, a_3, \ldots , a_q)$ is independent of the $p_i's$. Then, by Bayes' rule we have,

\begin{equation}
    \begin{split}
        & P_G(p_1, \ldots, p_q) \\
        & = \sum_{H \subseteq G}\sum_{(a_3, \ldots, a_q)}N(H, a_3, \ldots , a_q) \prod_{i=3}^{r}p_{i}^{a_i}
        \times P_{H'}\left(\frac{p_1}{p_1+p_2},\frac{p_2}{p_1+p_2}\right)\times (p_1 + p_2)^{|V(H')|},
    \end{split}
\end{equation}
where $H'$ is the subgraph induced by the remaining vertices, that is, $V(H') = V(G) \setminus V(H)$. To show that $P_G$ is schur-concave it suffices to show that
\begin{equation}
     P_{H'}\left(\frac{p_1}{p_1+p_2}, \frac{p_2}{p_1+p_2}\right)(p_1 + p_2)^{|V(H')|}
\end{equation}
is maximized when $p_1 = p_2$. To see this note that $H'$ is also claw-free since removing vertices keeps a claw-free graph claw-free. If $H'$
is not bipartite then it cannot have a proper coloring with 2 colors. Hence for the above term to be non-zero $H'$ must be a claw-free
bipartite graph. The only connected claw-free bipartite graphs are cycles of even length and paths. To see this, suppose $V(H')$ is partitioned
into sets $A,B$ such that there are no edges lying entirely inside $A$ or $B$. So for any $v \in A$ all its neighbors lie in $B$. Suppose $v$
has three neighbors $v_1, v_2, v_3$. Since they are all in $B$ there are no edges between them. This leads to a claw on $v, v_1, v_2, v_3$ and
contradicts the fact that $H'$ is claw-free. Hence the maximum degree of $H'$ is 2, thus implying that $H'$ is a disjoint union of cycles and
paths. Further since $H'$ is bipartite the cycles can only have even length. Suppose $H''$ is a connected component of $H'$. Then, $H''$ is either a path or an even cycle, hence,

\begin{equation}
    \begin{split}
        & P_{H''}\left(\frac{p_1}{p_1+p_2}, \frac{p_2}{p_1+p_2}\right)(p_1 + p_2)^{|V(H'')|}\\ \\
        & = 2(p_1p_2)^k \text{ or } p_1^kp_2^{k+1} +  p_2^kp_1^{k+1},
    \end{split}
\end{equation}
depending on whether $H''$ has $2k$ or $2k+1$ vertices. In both cases this is maximized when $p_1 = p_2$ (under the constraint that $p_1 + p_2$ is fixed). By multiplicativity, the same follows for $P_{H'}$. So,
\begin{equation}
    P(p_1, \ldots, p_q) \leq P\left(\frac{p_1+p_2}{2}, \frac{p_1+p_2}{2}, p_3, \ldots, p_q\right)
\end{equation}
and by symmetry,
\begin{equation}
    P(p_1, \ldots, p_q) \leq P\left(p_1, \ldots, p_{i-1},\frac{p_i+p_j}{2}, p_{i+1}, \ldots, p_{j-1}, \frac{p_i+p_j}{2}, p_{j+1}, \ldots, p_q\right),
\end{equation}
which proves that $P_G$ is schur concave. The above inequality attains equality if $p_i = p_j$. Further, since $P_{G}$ is a continuous function
on a compact space, a maximum is attained and by the above inequality it must be attained when $p_1= \cdots =p_q = 1/q$.

This completes the proof.

\end{proof}

\section{Acknowledgement}

The author would like thank Prof. Persi Diaconis for his guidance and for telling her about the shameful conjecture and related problems. Many thanks to P\'eter Csikv\'ari for numerous helpful discussions and to the anonymous referees for their detailed feedback and especially for the current, much shorter proof of Theorem \ref{Shameful}. 

\bibliography{SukhadaThesis}{}
\end{document}